\documentclass[reqno]{amsart}

\usepackage{graphicx,subfigure,color}

\usepackage{relsize,exscale}

\numberwithin{equation}{section}

\usepackage[latin1]{inputenc}
\usepackage[english]{babel}

\usepackage{amsmath,amsthm,amsfonts,latexsym,amssymb}
\usepackage[colorlinks]{hyperref}
\hypersetup{linkcolor=blue,citecolor=blue,filecolor=black,urlcolor=blue}
\usepackage{comment}

\usepackage{color}
{ \theoremstyle{plain}
\newtheorem{theorem}{Theorem}[section]
\newtheorem{proposition}[theorem]{Proposition}
\newtheorem{lemma}[theorem]{Lemma}
\newtheorem{corollary}[theorem]{Corollary}
  \theoremstyle{remark}
\newtheorem{remark}[theorem]{Remark}
  \theoremstyle{definition}
\newtheorem{definition}[theorem]{Definition}
}

\begin{document}


\title[]{The Soap Bubble Theorem and a $p$-Laplacian overdetermined problem}

\author[F. Colasuonno]{Francesca Colasuonno}
\address{Francesca Colasuonno\newline\indent
Dipartimento di Matematica ``Giuseppe Peano''\newline\indent
Universit\`a degli Studi di Torino\newline\indent
Via Carlo Alberto, 10\newline\indent
10123 Torino, Italy
}
\email{francesca.colasuonno@unito.it}

\author[F. Ferrari]{Fausto Ferrari}
\address{
Fausto Ferrari
\newline \indent 
Dipartimento di Matematica\newline\indent
Alma Mater Studiorum Universit\`a di Bologna\newline\indent
piazza di Porta S. Donato, 5\newline\indent
40126 Bologna, Italy
}
\email{fausto.ferrari@unibo.it}

\date{\today}

\begin{abstract} 
We consider the $p$-Laplacian equation $-\Delta_p u=1$ for $1<p<2$, on a regular bounded domain $\Omega\subset\mathbb R^N$, with $N\ge2$, under homogeneous Dirichlet boundary conditions. In the spirit of Alexandrov's Soap Bubble Theorem and of Serrin's symmetry result for the overdetermined problems, we prove that if the mean curvature $H$ of $\partial\Omega$ is constant, then $\Omega$ is a ball and the unique solution of the Dirichlet $p$-Laplacian problem is radial. The main tools used are integral identities, the $P$-function, and the maximum principle. 
\end{abstract}

\maketitle

\section{Introduction}
The celebrated Alexandrov's Soap Bubble Theorem \cite{Alex}, dated back to 1958, states that if $\Gamma$ is a compact hypersurface, embedded in $\mathbb R^N$, having constant mean curvature, then $\Gamma$ is a sphere. On the other hand, Serrin's symmetry result (1971) \cite{Serrin} for the following overdetermined problem
\begin{equation}\label{P2}
-\Delta u=1\quad\mbox{in }\Omega,\qquad u=0\quad\mbox{on }\partial \Omega,\\
\end{equation}
\begin{equation}\label{over2}
u_\nu=c\quad\mbox{on }\partial\Omega, 
\end{equation}
where $\Omega\subset\mathbb R^N$ is a bounded domain and $u_\nu$ is the outer normal derivative, states that if \eqref{P2}--\eqref{over2} has a solution, then $\Omega$ must be a ball, and the unique solution $u$ must be radial. It is nowadays well-known that these two results are strictly related. Indeed, for his proof, Serrin adapted to the PDEs the reflection principle, a geometrical technique introduced by Alexandrov in \cite{Alex}, and combined it with the maximum principle, giving rise to a very powerful and versatile tool, the {\it moving plane method}. This method is still very much used, since it can be successfully applied to a large class of PDEs. Besides the common techniques used, the link between these two results has been further highlighted by Reilly in \cite{Reilly}, where the author proposed an alternative proof of the Soap Bubble Theorem, considering the hypersurface $\Gamma$ as a level set (i.e., $\partial\Omega$) of the solution of \eqref{P2}. For his proof, Reilly found and exploited a relation between the Laplacian operator and the geometrical concept of mean curvature. 
Interestingly enough, Serrin's result for the overdetermined problem has been proved via a different technique  by Weinberger in a two-page paper \cite{Wein} that was published in the same volume of the same journal as the paper by Serrin \cite{Serrin}. Weinberger's proof is much simpler, it relies on some integral indentities, the maximum principle, and the introduction of an auxiliary function, the so-called $P$-function. Even if Weinberger's technique is less flexible than the moving plane method, it lends itself well to being re-read in quantitative terms. Recently, Magnanini and Poggesi in \cite{MP2016,MP2} proved the stability both for the Alexandrov's Soap Bubble theorem and for Serrin's result, by estimating the terms involved in an integral identity proved in \cite{Wein} and refined in \cite{PaySch}. Also the moving plane method has been reformulated in a quantitative version to get the stability of both Serrin's result, cf. \cite{ABR}, and Alexandrov's Theorem, cf. \cite{CV2016}.
In those stability results, the idea is to measure how much $\Omega$ is close to being a ball by estimating from above the difference $r_e-r_i$ ($r_e$ and $r_i$ being the radii of two suitable balls such that $B_{r_e}\subset\Omega\subset B_{r_i}$) in terms of the deviation of the normal derivative $u_\nu$ from being constant on $\partial\Omega$, or in terms of the deviation of the mean curvature $H$ from being constant on $\Gamma$. Other stability issues for the Serrin problem have been treated in \cite{BNST}. 

Serrin's symmetry result has been extensively studied and generalized also to the case of quasilinear problems. For the $p$-Laplacian operator $\Delta_p u=\mathrm{div}(|\nabla u|^{p-2}\nabla u)$, $1<p<\infty$, it has been proved that if the following problem
\begin{equation}\label{overP}
\begin{cases}
-\Delta_p u=1\quad&\mbox{in }\Omega,\\
u=0&\mbox{on }\partial \Omega,\\
|\nabla u|=c&\mbox{on }\partial\Omega\quad\mbox{for some }c>0 
\end{cases}
\end{equation}
admits a weak solution in the bounded domain $\Omega\subset\mathbb R^N$, then $\Omega$ is a ball. Garofalo and Lewis \cite{GL1989} proved this result via Weinberger's approach; Brock and Henrot \cite{BH} proposed a different proof via Steiner symmetrization for $p\ge 2$; Damascelli and Pacella \cite{DP} succeeded in adapting the moving plane method to the case $1<p<2$. Later, many other refinements and generalizations to more general operators have been proposed, we refer for instance to \cite{FGK2006,FK2008,Brock} and the references therein. 

In this paper, we consider the following Dirichlet $p$-Laplacian problem 
\begin{equation}\label{P}
\begin{cases}
-\Delta_p u=1\quad&\mbox{in }\Omega,\\
u=0&\mbox{on }\partial \Omega,
\end{cases}
\end{equation}
for $1<p<2$. Here $\Omega\subset\mathbb{R}^N$ is a smooth bounded domain and $N\ge2$. Due to its physical meaning, \eqref{P} is often referred to as $p$-torsion problem.  For this problem, existence and uniqueness of the solution can be easily proved via the Direct Method of the Calculus of Variations and using the strict convexity of the action functional associated, see Section \ref{sec2}. In the spirit of Reilly's result, we regard the hypersurface $\Gamma$ of Alexandrov's Theorem as the level set $\partial\Omega$ of the solution of \eqref{P} and we obtain, for smooth hypersurfaces, an alternative proof of the Soap Bubble Theorem. As a consequence, we prove the equivalence of the Soap Bubble Theorem to the Serrin-type symmetry result for the overdetermined problem \eqref{overP}, when $1<p<2$. 
We state here our main results. 

\begin{theorem}\label{SBT} Let $\Gamma\subset\mathbb R^N$ be a $C^{2,\alpha}$ surface which is the boundary of a bounded domain $\Omega\subset\mathbb R^N$, i.e. $\Gamma=\partial\Omega$, and denote by $H=H(x)$ the mean curvature of $\partial\Omega$. Suppose that $1<p<2$, that $u$ solves \eqref{P}, and that the set of critical points of $u$ has zero measure. Then the following statements are equivalent:
\begin{enumerate}
\item[a.] $\Omega$ is a ball;
\item[b.] $|u_\nu(x)|^{p-2}u_\nu(x)=-\frac{1}{NH(x)}$ for every $x\in\partial\Omega$;
\item[c.] $u$ is radial;
\item[d.] $H(x)=H_0$ for every $x\in\partial\Omega$.
\end{enumerate}
Moreover, if one of the previous ones holds, then 
\begin{enumerate}
\item[e.] $|\nabla u(x)|=\left(\frac{1}{NH_0}\right)^{\frac{1}{p-1}}$ for every $x\in\partial\Omega$.
\end{enumerate}
\end{theorem}

The implication d. $\Rightarrow$ a. in the previous theorem is a special case of the Soap Bubble Theorem of Alexandrov. We further observe that from the proof of the previous theorem, cf. formula \eqref{H0R0}, it results that if d. holds, then $\Omega$ must be a ball of radius $R_0=1/H_0$. Moreover, the fact that any of the statements a., b., c., or  d. implies e. is a simple consequence of the previous results, but we know that the converse implication e. $\Rightarrow$ a. holds as well: as proved in \cite{GL1989,FGK2006,FK2008}, the overdetermined problem \eqref{overP} admits a solution only if $\Omega$ is a ball of radius $R_0$. This allows us to state the equivalence of the Soap Bubble Theorem and of the Serrin-type result for the overdetermined $p$-Laplacian problem \eqref{overP} under suitable regularity assumptions, in the case $1<p<2$. 

\begin{corollary}\label{Alex-p-Serrin} Under the assumptions of Theorem \ref{SBT}, statements a., b., c., d., and e. are all equivalent.
\end{corollary}

Our proof technique takes inspiration from \cite{MP2016} and follows the approach of Weinberger. After having introduced the $P$-function \eqref{Pfunction} in terms of the solution of \eqref{P}, we derive the integral identity \eqref{identity} using the Divergence Theorem. The identity \eqref{identity} will be a key tool for the estimates in the rest of the paper. We recall then that the $p$-Laplacian of a smooth function can be expressed as the trace of a matrix-operator applied to the same function, cf. \eqref{p-lapl_traccia}, and we use a simple algebraic inequality \eqref{ineqFF} (known as Newton's inequality) to get an estimate of the $p$-Laplacian of a function. This suggests us to introduce in \eqref{integral} the integral $\mathcal I_p(u)$ which will play the role of the so-called Cauchy-Schwartz deficit in \cite{MP2016} for the linear case $p=2$. In view of Newton's inequality, the integral $\mathcal I_p(u)$ has a sign, it is always non-negative. Now the $P$-function comes into play: thanks to the fact that it satisfies a maximum principle, we can prove that, when $1<p<2$, $\mathcal I_p(u)$ vanishes only on radial solutions of \eqref{P}, cf. Lemma \ref{onlyrad}. This, combined with the integral identity \eqref{identity}, allows us to obtain an estimate from above of $\mathcal I_p(u)$ in terms of some boundary integrals involving only the mean curvature $H$ and the normal derivative $u_\nu$, see Theorem \ref{th:main}. Then Theorem \ref{SBT} and Corollary \ref{Alex-p-Serrin} are easy consequences: $\mathcal I_p(u)$ is zero (or equivalently the solution of \eqref{P} is radial) if and only if the mean curvature $H$ is constant on $\partial\Omega$ or the modulus of the gradient of $u$ is constant on $\partial\Omega$. Finally, in Corollary \ref{I-H}, we give an estimate from above of the integral $\mathcal I_p(u)$ in terms of the $L^1(\partial\Omega)$-norm of the deviation of $H$ from being constant and some constants which only depend on the geometry of the problem, cf. \eqref{IpH}.

The paper is organized as follows: in Section \ref{sec2} we introduce some useful notation, the $P$-function, some known results, and some preliminary lemmas. In Section \ref{sec3} we prove Theorem \ref{SBT} and its consequences, while in Section \ref{sec4}, we present some comments on the stability for the $p$-overdetermined problem.

\section{Preliminaries}\label{sec2}
We first introduce the main important quantities and notation involved. Throughout the paper, with abuse of notation, we use the symbol $|\cdot|$ to denote both the $N$-dimensional and the $(N-1)$-dimensional Lebesgue measures. We further denote by $\|\cdot\|$ the Frobenius matrix norm and by $\langle\cdot,\cdot\rangle$ the scalar product in $\mathbb R^N$.  \smallskip

\noindent {\bf The $p$-Laplacian on non-critical level sets of $u$.}
The $p$-Laplacian of a regular function $v$ can be expressed as follows
\begin{equation}\label{plapl}
\Delta_p v=|\nabla v|^{p-2}\left(\Delta v+(p-2)\frac{\langle D^2v\nabla v,\nabla v\rangle}{|\nabla v|^2}\right),
\end{equation}
where $D^2v$ denotes the Hessian matrix of $v$. Moreover, we recall that, in view of \eqref{plapl}, it is possible to express the $p$-Laplacian of any $C^2$-function $v$ as follows 
\begin{equation}\label{p-lapl_traccia}
\begin{aligned}
{\Delta_p v}&=|\nabla v|^{p-2}\left(\Delta v+(p-2)\langle D^2v\frac{\nabla v}{|\nabla v|},\frac{\nabla v}{|\nabla v|}\rangle\right)\\
&=|\nabla v|^{p-2}\left(\mathrm{Tr}(D^2 v)+\frac{p-2}{|\nabla v|^2}\sum_{i,j=1}^N \frac{\partial^2 v}{\partial x_i\partial x_j} \frac{\partial v}{\partial x_i}\frac{\partial v}{\partial x_j}\right)\\
&=|\nabla v|^{p-2}\left[\mathrm{Tr}(D^2 v)+(p-2)\mathrm{Tr}\left(\frac{\nabla v}{|\nabla v|}\otimes \frac{\nabla v}{|\nabla v|}\cdot D^2v\right)\right]\\
&=\mathrm{Tr}\left[|\nabla v|^{p-2}\left(\mathrm{I}+(p-2)\frac{\nabla v}{|\nabla v|}\otimes\frac{\nabla v}{|\nabla v|}\right)D^2v\right],
\end{aligned}
\end{equation}
where we have denoted simply by $\mathrm{I}$ the $N\times N$ identity matrix.
\smallskip

Let $u$ be a solution of \eqref{P}. We denote by $\nu$ the following vector field 
$$\nu=-\frac{\nabla u}{|\nabla u|},$$
which coincides with the external unit normal on $\partial \Omega$, being $u_{|_{\partial\Omega}}$ constant.
The mean curvature of the regular level sets of $u$ is given by 
$$H=-\frac1{N-1}\mathrm{div}\frac{\nabla u}{|\nabla u|}.$$

It is possible to see that, on non-critical level sets of $u$, the Laplacian of $u$ can be expressed in terms of $H$ as follows
\begin{equation}\label{Deltau}
\Delta u=u_{\nu\nu}+(N-1)Hu_\nu,
\end{equation}
where $u_{\nu}=\nabla u\cdot\nu=-|\nabla u|$ and $u_{\nu\nu}=\langle D^2u \,\nu,\nu\rangle$. Therefore, on non-critical level sets of $u$, we can write the $p$-Laplacian as
\begin{equation}\label{Deltapu}
\Delta_p u=|u_\nu|^{p-2}\left[(p-1)u_{\nu\nu}+(N-1)H u_\nu\right].
\end{equation}

\smallskip

\noindent {\bf The $P$-function.}
In terms of a solution $u$ of \eqref{P}, we can define the so-called $P$-function as
\begin{equation}\label{Pfunction}
P:=\frac{2(p-1)}{p}|\nabla u|^p+\frac{2}{N}u\quad\mbox{a.e. in }\Omega,
\end{equation}
we refer to \cite[Chapter 7, formula (7.6) with $v(q)=q^{\frac{p-2}{2}}$ and $q=|\nabla u|^2$]{S1981} for its derivation. The main feature of $P$ is that it satisfies a maximum principle, which is the starting point for finding useful bounds for the main quantities involved in this problem.

\begin{definition}
Let $\Omega\subset \mathbb R^N$ be a bounded domain. $\Omega$ satisfies the {\it interior sphere condition} if for every $x\in\partial\Omega$ there exist $x_0\in\Omega$ and $r>0$ such that $B_r(x_0):=\{y\in\mathbb R^N\,:\, |y-x_0|<r\}\subset\Omega$ and $x\in\partial B_r(x_0)$.
\end{definition}

We recall that if $\Omega$ is a $C^2$ bounded domain, then it satifies the interior sphere condition. 

\begin{lemma}\label{Pnu>0} Let $\Omega$ be of class $C^{1,\alpha}$ and satisfy the interior sphere condition.
If $u$ solves \eqref{P}, then $P$ is either constant in $\bar\Omega$ or it satisfies $P_\nu>0$ on $\partial\Omega$.
\end{lemma}
\begin{proof} The proof of this lemma is given in \cite[Lemma 3.2]{FGK2006} for a solution of the overdetermined problem \eqref{overP}; we report the outline of the proof here in order to highlight that it continues to hold even if $u$ does not satisfy $|\nabla u|=\mathrm{const.}$ on $\partial\Omega$.

Since $u$ solves \eqref{P}, then by \cite[Theorem 3.2.2]{PucciSerrin}, $u\ge 0$ a.e. in $\Omega$ and by \cite[Theorem 1]{Lieberman}, $u$ is of class $C^{1,\alpha}(\bar\Omega)$. Now, \cite[Theorem 5]{Vazquez} guarantees that $|\nabla u|\ge \max_{\partial\Omega}|\nabla u|>0$ on $\partial\Omega$. By continuity, $|\nabla u|\not= 0$ in a closed neighborhood $D\subset\bar\Omega$ of $\partial\Omega$.

Now, suppose that $P$ is not constant in $\bar \Omega$. Under this assumption, as in \cite[Lemma~3.2 - Claim -  Step 2]{FGK2006}, it is possible to prove that $P$ attains its maximum on $\partial \Omega$ and that, if $P$ also attains its maximum at a point $\bar x\in\Omega$, then necessarily $\nabla u(\bar x)=0$. 
Therefore, being $D\subset\bar\Omega$ a closed neighborhood of $\partial\Omega$, $P$ attains its maximum in $D$ only on $\partial\Omega$. By the proof of \cite[Lemma 3.2]{FGK2006}, we know that $P$ satisfies in $D$ a uniformly elliptic equation and so it satisfies the classical Hopf's lemma. Hence, $P_\nu>0$ on $\partial\Omega$.
\end{proof}

For future use, we derive here an easy identity holding true for any $u$ solution of \eqref{P}. By integration by parts, the Divergence Theorem, and \eqref{Deltau} we get
\begin{equation}\label{id1}
\begin{aligned}
\int_\Omega\langle|\nabla u|^{p-2}\nabla u,\nabla \Delta u\rangle dx &=-\int_\Omega \Delta_p u\Delta u dx+\int_{\partial \Omega}\Delta u|\nabla u|^{p-2}\nabla u\cdot\nu d\sigma\\
&=\int_\Omega\Delta u dx+\int_{\partial\Omega}\Delta u |u_\nu|^{p-2}u_\nu d\sigma\\
&=\int_{\partial\Omega}u_\nu d\sigma +\int_{\partial \Omega}|u_\nu|^{p-2}u_\nu[u_{\nu\nu}+(N-1)Hu_\nu]d\sigma\\
&=\int_{\partial\Omega}u_\nu d\sigma -\int_{\partial \Omega}|u_\nu|^{p-1}u_{\nu\nu}d\sigma+(N-1)\int_{\partial\Omega}H|u_\nu|^p d\sigma,
\end{aligned}
\end{equation}
where we used that $\partial\Omega$ is a non-critical level set of $u$, as showed in the proof of Lemma \ref{Pnu>0}.\smallskip 

\noindent {\bf Reference constant mean curvature and reference domain.} We introduce here some reference geometric constants which are related to problem \eqref{P}.
These constants will be useful to compare problem \eqref{P} with the same problem set in a ball instead of a general domain $\Omega$. 

By Minkowski's identity, i.e., 
$$
\int_{\partial\Omega}H(x)\langle x-z,\nu(x)\rangle d\sigma=|\partial\Omega|\quad\mbox{for any }z\in\mathbb R^N,
$$
we get,  by the Divergence Theorem and if $H$ is constant:
$$
|\partial\Omega|=H\int_{\partial\Omega}\langle x-z,\nu(x)\rangle d\sigma=\int_\Omega\sum_{i=1}^N\frac{\partial (x-z)}{\partial x_i} dx=H|\Omega| N.
$$
If $H$ is not constant, we can take as reference constant mean curvature the quantity 
$$H_0:=\frac{|\partial\Omega|}{N|\Omega|}$$
and, as reference domain, a ball of radius 
$$R_0=\frac{1}{H_0}=\frac{N|\Omega|}{|\partial\Omega|}.$$
\smallskip 

\noindent {\bf Existence and uniqueness for \eqref{P}.}
Problem \eqref{P} has a variational structure with associated action functional $I:W^{1,p}_0(\Omega)\to\mathbb R$ given by 
$$
I(u):=\int_\Omega\left (\frac{1}{p}|\nabla u|^p - u \right) dx.
$$
By strict convexity and the Direct Method of Calculus of Variations, it is possible to prove that $I$ has a unique minimizer. Hence, \eqref{P} has a unique weak solution $u\in W^{1,p}_0(\Omega)$. 

From now on in the paper, we denote by $\mathcal C$ the critical set of the solution $u$ of problem \eqref{P}, namely 
$$\mathcal C:=\{x\in\Omega\,:\,|\nabla u (x)|=0\}.$$

By \cite[Lemma 3.1]{FGK2006}, we know that the solution $u$ of \eqref{P} is of class $C^{2,\alpha}(\bar\Omega\setminus \mathcal C)$.

Therefore, hereafter we assume that $\Omega$ is of class $C^{2,\alpha}$ in order to guarantee that the solution $u$ of \eqref{P} is of class $C^{2,\alpha}$ in a neighborhood of $\partial\Omega$ (this is a consequence of the regularity of $u$ and of the first part of the proof of Lemma \ref{Pnu>0}).

\begin{lemma}\label{Identity-Pnu>0} 
Let $u$ solve \eqref{P} and suppose that its critical set $\mathcal C$ has zero $N$-dimensional measure. The following identity holds
\begin{equation}\label{identity}
\begin{aligned}
\int_\Omega&\left\{(p-1)|\nabla u|^{p-2}\left[(p-2)\left\|D^2u\frac{\nabla u}{|\nabla u|}\right\|^2+\|D^2u\|^2+\langle\nabla u,\nabla \Delta u\rangle\right]+\frac{\Delta u}{N}\right\}dx\\
&\hspace{5cm}=-\int_{\partial\Omega}(N-1)\left(\frac{1}{N}u_\nu+H|u_\nu|^p\right)d\sigma
\end{aligned}
\end{equation}
\end{lemma}
\begin{proof}
By straightforward calculations, we get 
\begin{equation}\label{Pnu}
P_\nu=\nabla P\cdot\nu=2 u_\nu\left((p-1)|u_\nu|^{p-2}u_{\nu\nu}+\frac1N\right),
\end{equation}
cf. \cite[formula (7.7)]{S1981} with $f\equiv w\equiv1$, $\alpha=2/N$, $q=|\nabla u|^2$, and $v(q)=q^{(p-2)/2}$.
By taking into account \eqref{Deltau}, \eqref{Deltapu}, and the equation in \eqref{P}, we can rewrite $P_\nu$ as
\begin{equation}\label{Pnu1}
\begin{aligned}
P_\nu&=2 u_\nu\left(\Delta_p u-(N-1)H |u_\nu|^{p-2}u_{\nu}+\frac1N\right)\\
&=-2(N-1)\left(\frac{1}{N}u_\nu+H|u_\nu|^p\right). 
\end{aligned}
\end{equation}
Moreover,
\begin{equation}\label{DeltaP}
\Delta P=2\left\{(p-1)|\nabla u|^{p-2}\left[(p-2)\left\|D^2 u\frac{\nabla u}{|\nabla u|}\right\|^{2}+\|D^2u\|^2+\langle\nabla u,\nabla\Delta u\rangle\right]+\frac{\Delta u}N\right\}
\end{equation}
cf. \cite[formula (7.9)]{S1981}.
The conclusion then follows, since $\int_\Omega \Delta P dx=\int_{\partial\Omega}P_\nu d\sigma$,  by the Divergence Theorem. 
\end{proof}

\begin{proposition}[Newton's inequality]\label{matrix_ineq}
Let $n\in\mathbb N$ and $A$ be a $(n\times n)$-matrix, then 
\begin{equation}\label{ineqFF}
\|A\|^2\ge\frac{(\mathrm{Tr}(A))^2}{n},
\end{equation}
where denotes $\mathrm{Tr}(\cdot)$ the trace of a matrix. Furthermore, the equality holds in \eqref{ineqFF} if and only if $A=k\mathrm{I}_n$ for some constant $k$.
\end{proposition}
\begin{proof} The proof is standard, but we report it here for the sake of completeness.
The statement is trivial for $n=1$. We proceed by induction on $n\ge2$. If we denote by $a_{ij}$ the elements of the matrix $A$, we obtain for $n=2$ that 
\begin{equation}\label{n=2}
(\mathrm{Tr}(A))^2=(a_{11}+a_{22})^2=a_{11}^2+a_{22}^2+2a_{11}a_{22}\le2(a_{11}^2+a_{22}^2)\le 2 \|A\|^2,
\end{equation}
where we have used that $2a_{11}a_{22}\le a_{11}^2+a_{22}^2$, being $(a_{11}-a_{22})^2=a_{11}^2+a_{22}^2-2a_{11}a_{22}\ge 0$. As a consequence, we observe that \eqref{n=2} holds with the equality signs if and only if $a_{11}=a_{22}$ and $a_{12}=a_{21}=0$. 
We now assume that \eqref{ineqFF} holds true for $n$ and we prove it for $n+1$. Indeed, 
\begin{equation}\label{induc-step}
\begin{aligned}
(\mathrm{Tr}(A))^2&=\left(\sum_{i=1}^{n+1}a_{ii}\right)^2=\left(\sum_{i=1}^{n}a_{ii}+a_{n+1,n+1}\right)^2\\
&=\left(\sum_{i=1}^{n}a_{ii}\right)^2+2\left(\sum_{i=1}^{n}a_{ii}\right)a_{n+1,n+1}+a_{n+1,n+1}^2\\
&\le n\sum_{i=1}^{n}a_{ii}^2+n\sum_{\underset{i\neq j}{i,\,j=1}}^na_{ij}^2+2\left(\sum_{i=1}^{n}a_{ii}\right)a_{n+1,n+1}+a_{n+1,n+1}^2.
\end{aligned}
\end{equation}
Now, as above, we can estimate 
$$
\begin{aligned}
2\left(\sum_{i=1}^{n}a_{ii}\right)a_{n+1,n+1}&=\sum_{i=1}^{n}2a_{ii}a_{n+1,n+1}\\
&\le\sum_{i=1}^{n}(a_{ii}^2+a_{n+1,n+1}^2)=na_{n+1,n+1}^2+\sum_{i=1}^{n}a_{ii}^2,
\end{aligned}
$$
where the equality is achieved only for $a_{ii}=a_{n+1,n+1}$ for every $i=1,\dots,n$. Therefore, combining this estimate with \eqref{induc-step}, we obtain
$$
\begin{aligned}
(\mathrm{Tr}(A))^2&\le n\sum_{i=1}^{n}a_{ii}^2+a_{n+1,n+1}^2+na_{n+1,n+1}^2+\sum_{i=1}^{n}a_{ii}^2+n\sum_{\underset{i\neq j}{i,\,j=1}}^na_{ij}^2\\
&=(n+1)\sum_{i=1}^{n}a_{ii}^2+(n+1)a_{n+1,n+1}^2+n\sum_{\underset{i\neq j}{i,\,j=1}}^na_{ij}^2\\
&=(n+1)\sum_{i=1}^{n+1}a_{ii}^2+n\sum_{\underset{i\neq j}{i,\,j=1}}^na_{ij}^2\le (n+1)\sum_{i,\,j=1}^n a_{ij}^2,
\end{aligned}
$$
where the equalities hold only when $A=k\mathrm{I}_{n+1}$ for some constant $k$, and the proof is complete.
\end{proof}

\begin{corollary}
Let $v$ be any $C^2$-function, then the following inequality holds
\begin{equation}\label{ineq_FF_deltap}
(\Delta_p v)^2\le N|\nabla v|^{2(p-2)}\left\|\left(\mathrm{I}+(p-2)\frac{\nabla v}{|\nabla v|}\otimes \frac{\nabla v}{|\nabla v|}\right)D^2v\right\|^2.
\end{equation}
\end{corollary}
\begin{proof} 
Taking into account \eqref{p-lapl_traccia}, it is enough to apply Proposition \ref{matrix_ineq} with $n:=N$ and $A:=|\nabla v|^{p-2}\left(\mathrm{I}+(p-2)\frac{\nabla v}{|\nabla v|}\otimes\frac{\nabla v}{|\nabla v|}\right)D^2v$.
\end{proof}
\smallskip

For every $z\in\mathbb R^N$ and $r>0$, we introduce the function
\begin{equation}\label{rad-sol}
w_r(x):=-\frac{p-1}{pN^{\frac1{p-1}}}\left(|x-z|^{\frac{p}{p-1}}-r\right)\quad\mbox{for every }x\in\Omega.
\end{equation}
We observe that, if $z\in\Omega$ and $p>2$, $w$ does not have $C^2$ partial derivatives. 
Clearly, $w_r$ is radial about $z$, and, if $\Omega=B_r(z)$, it solves \eqref{P}.
Indeed, by straightforward calculations we get 
$$
\begin{aligned}
&\nabla w_r=-N^{-\frac{1}{p-1}}|x-z|^{\frac{p}{p-1}-2}(x-z),\\
&|\nabla w_r|^{p-2}\nabla w_r=-\frac1N(x-z),
\end{aligned}$$
and so 
$$
\Delta_p w_r=\mathrm{div}\left(-\frac1N(x-z)\right)=-1.
$$

We are now ready to prove the following result. 

\begin{lemma}\label{onlyrad}
\begin{itemize} Let  $1<p<2$, then the following statements hold true.
\item[(i)] Let $w_r$ be defined as in \eqref{rad-sol}, then for $v:=w_r$ the equality holds in \eqref{ineq_FF_deltap}.
\item[(ii)] Let $u$ solve \eqref{P}. Suppose that the critical set $\mathcal C$ of $u$ has zero $N$-dimensional measure and that for $v:=u$ the equality holds in \eqref{ineq_FF_deltap} for every $x\in\Omega\setminus\mathcal C$. Then $u$ is radial.
\end{itemize}  
\end{lemma} 
\begin{proof}
(i) Since 
$$
\frac{\partial^2 w_r}{\partial x_i \partial x_j}=-N^{-\frac{1}{p-1}}\left[\frac{2-p}{p-1}|x-z|^{\frac{p}{p-1}-4}(x_j-z_j)(x_i-z_i)+\delta_{ij}|x-z|^{\frac{p}{p-1}-2}\right],
$$
the Hessian of $w_r$ has the following expression
$$
D^2w_r=-N^{-\frac{1}{p-1}}|x-z|^{\frac{2-p}{p-1}}\left(\frac{2-p}{p-1}\cdot\frac{x-z}{|x-z|}\otimes\frac{x-z}{|x-z|}+\mathrm{I}\right).
$$
By
$$
\left(\frac{x-z}{|x-z|}\otimes\frac{x-z}{|x-z|}\right)^2=\frac{x-z}{|x-z|}\otimes\frac{x-z}{|x-z|}\quad\mbox{ and }\quad \frac{\nabla w_r}{|\nabla w_r|}=\frac{x-z}{|x-z|},
$$
we get 
$$
\begin{aligned}
&|\nabla w_r|^{p-2}\left(\mathrm{I}+(p-2)\frac{\nabla w_r}{|\nabla w_r|}\otimes\frac{\nabla w_r}{|\nabla w_r|}\right)D^2w_r\\
&=-\frac{|x-z|^{\frac{2-p}{p-1}+(\frac{p}{p-1}-1)(p-2)}}{N}\left[\mathrm{I}+\left(\frac{2-p}{p-1}-\frac{(p-2)^2}{p-1}+p-2\right)\frac{x-z}{|x-z|}\otimes\frac{x-z}{|x-z|}\right]\\
&=-\frac{1}{N}\mathrm{I}.
\end{aligned}
$$
Hence, by Proposition \ref{matrix_ineq}, \eqref{ineq_FF_deltap} holds with the equality sign for $v:=w_r$.\smallskip

(ii) By Proposition \ref{matrix_ineq}, we know that the equality holds in \eqref{ineq_FF_deltap} if and only if 
$$
|\nabla u|^{p-2}\left(\mathrm{I}+(p-2)\frac{\nabla u}{|\nabla u|}\otimes \frac{\nabla u}{|\nabla u|}\right)D^2u=k \mathrm{I}
$$ 
for some constant $k$. By $\left\|(2-p)\frac{\nabla u}{|\nabla u|}\otimes \frac{\nabla u}{|\nabla u|}\right\|=|2-p|<1$, 
$$
\mathrm{det}\left(\mathrm{I}-(2-p)\frac{\nabla u}{|\nabla u|}\otimes \frac{\nabla u}{|\nabla u|}\right)\neq 0,
$$
and 
$$
\left(\frac{x}{|x|}\otimes\frac{x}{|x|}\right)^i=\frac{x}{|x|}\otimes\frac{x}{|x|}\quad\mbox{for all }x\in\mathbb R^N\mbox{ and all }i\in\mathbb N,
$$
we get on $\Omega\setminus\mathcal C$ 
\begin{equation}\label{system}
\begin{aligned}
D^2u&=\frac{k}{|\nabla u|^{p-2}}\left(\mathrm I-(2-p)\frac{\nabla u}{|\nabla u|}\otimes \frac{\nabla u}{|\nabla u|}\right)^{-1}\\
&=\frac{k}{|\nabla u|^{p-2}}\sum_{i=0}^\infty(2-p)^i\left(\frac{\nabla u}{|\nabla u|}\otimes\frac{\nabla u}{|\nabla u|}\right)^i\\
&=\frac{k}{|\nabla u|^{p-2}}\left(\mathrm{I}+\frac{\nabla u}{|\nabla u|}\otimes\frac{\nabla u}{|\nabla u|}\sum_{i=1}^\infty(2-p)^i\right)\\
&=\frac{k}{|\nabla u|^{p-2}}\left[\mathrm{I}+\frac{\nabla u}{|\nabla u|}\otimes\frac{\nabla u}{|\nabla u|}\left(\frac{1}{1-(2-p)}-1\right)\right]\\
&=\frac{k}{|\nabla u|^{p-2}}\left(\mathrm{I}-\frac{p-2}{p-1}\frac{\nabla u}{|\nabla u|}\otimes\frac{\nabla u}{|\nabla u|}\right).
\end{aligned}
\end{equation}
Namely, for $i,\,j=1,\dots,N$
\begin{equation*}
\partial^2_{ij}u=\frac{k}{|\nabla u|^{p-2}}\left(\delta_{ij}-\frac{p-2}{p-1}\frac{\partial_i u\partial_j u}{|\nabla u|^2}\right). 
\end{equation*}
Hence, in particular, 
\begin{equation}\label{Laplacian}
\Delta u=\frac{k}{|\nabla u|^{p-2}}\sum_{i=1}^N\left(1-\frac{p-2}{p-1}\frac{(\partial_i u)^2}{|\nabla u|^2}\right)=\frac{k}{|\nabla u|^{p-2}}\left(N-\frac{p-2}{p-1}\right).
\end{equation}

Furthermore, since $u$ solves \eqref{P}, then by \eqref{system}, \eqref{Laplacian}, and \eqref{plapl}, we have
$$
\begin{aligned}
-1&=|\nabla u|^{p-2}\left(\Delta u+(p-2)\left\langle  D^2u\frac{\nabla u}{|\nabla u|}, \frac{\nabla u}{|\nabla u|}\right\rangle\right)\\
&=k\sum_{i=1}^N\left(1-\frac{p-2}{p-1}\frac{(\partial_i u)^2}{|\nabla u|^2}\right)+(p-2)\left\langle|\nabla u|^{p-2}D^2 u\frac{\nabla u}{|\nabla u|},\frac{\nabla u}{|\nabla u|}\right\rangle\\
&=k\sum_{i=1}^N\left(1-\frac{p-2}{p-1}\frac{(\partial_i u)^2}{|\nabla u|^2}\right)+(p-2)\left\langle k\left(I-\frac{p-2}{p-1}\frac{\nabla u}{|\nabla u|}\otimes\frac{\nabla u}{|\nabla u|}\right) \frac{\nabla u}{|\nabla u|},\frac{\nabla u}{|\nabla u|}\right\rangle\\
&=k\left[N-\frac{p-2}{p-1}+(p-2)\left(1-\frac{p-2}{p-1}\left\langle\frac{\nabla u}{|\nabla u|}\otimes\frac{\nabla u}{|\nabla u|} \frac{\nabla u}{|\nabla u|},\frac{\nabla u}{|\nabla u|}\right\rangle\right)\right]\\
&=k\left[N-\frac{p-2}{p-1}+(p-2)\left( 1-\frac{p-2}{p-1}\right)\right]=kN,
\end{aligned}
$$
where in the last equality, we have used that 
$$
\frac{x}{|x|}\otimes\frac{x}{|x|}\frac{x}{|x|}=\frac{x}{|x|} \quad\mbox{for all }x\in\mathbb R^N.
$$
Hence, $k=-\frac{1}{N}$.

Now, by the equation in \eqref{P}, \eqref{Laplacian}, and \eqref{Deltau}, we get on non-critical level sets of $u$ 
$$
\begin{gathered}
|u_\nu|^{p-2}\left[(p-1)u_{\nu\nu}+(N-1)Hu_\nu\right]=-1,\\
u_{\nu\nu}+(N-1)Hu_\nu=\left(\frac{p-2}{N(p-1)}-1\right)\frac{1}{|u_\nu|^{p-2}},
\end{gathered}
$$
being $u_\nu=-|\nabla u|$. These two identities give
$$
|u_\nu|^{p-2}u_{\nu\nu}=-\frac{1}{N(p-1)}
$$
and consequently
\begin{equation}\label{Hunu}
H=\frac{1}{N|u_\nu|^{p-1}}\quad\mbox{on }\partial\Omega.
\end{equation}
Now, by Lemma \ref{Pnu>0}, we know that either $P$ is constant on $\bar\Omega$, or $P_\nu>0$ on $\partial\Omega$. 
If the first case occurs, then it is possible to see that all level sets of $u$ are isoparametric surfaces. In particular, since $u$ satisfies homogeneous Dirichlet boundary conditions, all level sets must be concentric spheres and so $u$ is radial, cf. \cite[Remark 5.5]{FGK2006} and \cite[Theorem 5]{Kawohl1999}.
If the second case occurs, then by \eqref{Pnu1}, 
$$\frac{1}{N}u_\nu+H|u_\nu|^p<0\quad\mbox{on }\partial\Omega,$$
therefore, by \eqref{Hunu},
$$
0=\frac{1}{N}(u_\nu-u_\nu)=\frac{u_\nu}{N}+\frac{|u_\nu|}{N}<0\quad\mbox{on }\partial\Omega.
$$
This is impossible and concludes the proof. 
\end{proof}

\section{Proof of the main results}\label{sec3}
Let $u$ solve \eqref{P} and suppose that its critical set $\mathcal C$ has zero $N$-dimensional measure. We introduce the following integral
\begin{equation}\label{integral} 
\mathcal I_p(u):=\int_\Omega\left[|\nabla u|^{(p-2)}\left\|\left(\mathrm{I}+(p-2)\frac{\nabla u}{|\nabla u|}\otimes \frac{\nabla u}{|\nabla u|}\right)D^2u\right\|^2-\left(\frac{\Delta_p u}{N^{1/2}|\nabla u|^{\frac{p-2}{2}}}\right)^2\right]dx.
\end{equation}

\begin{theorem}\label{th:main}
Let $1<p<2$ and $\partial\Omega$ be a $C^{2,\alpha}$ bounded domain of $\mathbb R^N$.
If $u$ solves \eqref{P} and has $|\mathcal C|=0$, then
\begin{itemize}
\item[(i)] $\mathcal I_p(u)\ge 0$ and $\mathcal I_p(u)= 0$ if and only if $u$ is radial; 
\item[(ii)] $\displaystyle{\mathcal I_p(u)\le  -\frac{p(N-1)}{p-1}\int_{\partial \Omega}\left(\frac{1}{N}u_\nu+H|u_\nu|^p\right)d\sigma}$;
\item[(iii)] $\displaystyle{\mathcal I_p(u)\le \frac{p(N-1)}{p-1}\int_{\partial\Omega}|u_\nu|^p(H_0-H)d\sigma}$. 
\end{itemize}
\end{theorem}
\begin{proof}[Proof of Theorem \ref{th:main}] (i) By \eqref{ineq_FF_deltap}, we know that $\mathcal I_p(u)\ge 0$ and, by Lemma \ref{onlyrad}, we know that $\mathcal I_p(u)=0$ if and only if $u$ is radial. 

(ii) First, we observe that a.e. in $\Omega$ we have
$$
\begin{aligned}
&\left\|\left(\mathrm{I}+(p-2)\frac{\nabla u}{|\nabla u|}\otimes \frac{\nabla u}{|\nabla u|}\right)D^2u\right\|^2=
\sum_{i,j=1}^N\left(\partial^2_{ij}u+(p-2)\sum_{k=1}^N\frac{\partial_iu}{|\nabla u|}\frac{\partial_ku}{|\nabla u|}\partial^2_{kj}u\right)^2\\
&=\|D^2u\|^2+2(p-2)\sum_{i,j=1}^N\partial^2_{ij}u\frac{\partial_iu}{|\nabla u|}\sum_{k=1}^N\frac{\partial_k u}{|\nabla u|}\partial^2_{kj}u+(p-2)^2\sum_{i,j=1}^N\left(\sum_{k=1}^N\frac{\partial_iu}{|\nabla u|}\frac{\partial_ku}{|\nabla u|}\partial^2_{kj}u\right)^2\\
&=\|D^2u\|^2+2(p-2)\sum_{j=1}^N\left(\sum_{i=1}^N\frac{\partial_i u}{|\nabla u|}\partial^2_{ij}u\right)^2+(p-2)^2\left\|D^2u\frac{\nabla u}{|\nabla u|}\right\|^2\\
&=\|D^2u\|^2+p(p-2)\left\|D^2u\frac{\nabla u}{|\nabla u|}\right\|^2.
\end{aligned}
$$
Furthermore, by \eqref{identity}, we get 
$$
\begin{aligned}
p(p-2)\int_\Omega |\nabla u|^{p-2}\left\|D^2u\frac{\nabla u}{|\nabla u|}\right\|^2dx=&-p\int_\Omega\left[|\nabla u|^{p-2}\left(\|D^2u\|^2+\langle\nabla u,\nabla \Delta u\rangle \right) + \frac{\Delta u}{N(p-1)}\right]dx \\
&-p\frac{N-1}{p-1}\int_{\partial \Omega}\left(\frac{1}{N}u_\nu+H|u_\nu|^p\right)d\sigma. 
\end{aligned}
$$
Hence, using these last two identities, we can rewrite $\mathcal I_p(u)$ as
$$
\begin{aligned}
\mathcal I_p(u)=&\int_{\Omega}\left\{|\nabla u|^{p-2}\left[-(p-1)\|D^2u\|^2-p\langle\nabla u,\nabla \Delta u\rangle\right]-\frac{p}{N(p-1)}\Delta u-\frac{(\Delta_pu)^2}{N|\nabla u|^{p-2}}\right\}dx\\
&-\frac{p(N-1)}{p-1}\int_{\partial \Omega}\left(\frac{1}{N}u_\nu+H|u_\nu|^p\right)d\sigma.
\end{aligned}
$$
On the other hand, by \eqref{plapl}, the $C^{2,\alpha}$ regularity of $u$ in a neighborhood of $\partial\Omega$, and the Divergence Theorem
$$
\begin{aligned}
-p\int_\Omega&\left(|\nabla u|^{p-2}\langle\nabla u,\nabla \Delta u\rangle+\frac{1}{N(p-1)}\Delta u\right) dx
\\&=\int_\Omega -p\left(1+\frac{1}{N(p-1)}\right)\Delta udx+p\int_{\partial\Omega}|\nabla u|^{p-1}\Delta u d\sigma
\\
&=p\int_{\partial\Omega}\left(1+\frac{1}{N(p-1)}\right)|\nabla u|(1+|\nabla u|^{p-2}\Delta u)d\sigma
\\
&=-p(p-2)\left(1+\frac{1}{N(p-1)}\right)\int_{\partial\Omega}|\nabla u|^{p-1}\langle D^2u\frac{\nabla u}{|\nabla u|},\frac{\nabla u}{|\nabla u|} \rangle d\sigma.
\end{aligned}
$$
Hence,
\begin{equation}\label{identityI}
\begin{aligned}
\mathcal I_p(u) = &\int_{\Omega}\left\{-(p-1)|\nabla u|^{p-2}\|D^2u\|^2-\frac{(\Delta_pu)^2}{N|\nabla u|^{p-2}}\right\}dx\\
& -\frac{p(N-1)}{p-1}\int_{\partial \Omega}\left(\frac{1}{N}u_\nu+H|u_\nu|^p\right)d\sigma \\
&-p(p-2)\left(1+\frac{1}{N(p-1)}\right)\int_{\partial\Omega}|\nabla u|^{p-1}\langle D^2u\frac{\nabla u}{|\nabla u|},\frac{\nabla u}{|\nabla u|} \rangle d\sigma.
\end{aligned}
\end{equation}

In order to estimate from above $\mathcal I_p(u)$, we want to determine the sign of the last term in \eqref{identityI}. 
By Lemma \ref{Pnu>0}, we know that either $P_\nu>0$ on $\partial\Omega$ or $P$ is constant in $\bar\Omega$. If the second case occurs, then, as in the proof of Lemma \ref{onlyrad}-(ii), all level sets of $u$ are concentric spheres, and in particular $\Omega$ is a ball. Without loss of generality we can suppose $\Omega$ to be a ball centered in the origin $B_r$, thus, the unique solution of \eqref{P} is $w_r$, given in \eqref{rad-sol}, with $z=0$. Then, by straightforward calculations, we have for every $x\in\partial B_r$
$$
H(x)=-\frac{1}{N-1}\mathrm{div}\frac{\nabla w_r}{|\nabla w_r|}=\frac{1}{N-1}\sum_{i=1}^N\left(\frac{1}{|x|}-\frac{x_i^2}{|x|^3}\right)=\frac{1}{r}
$$
and $$(w_r)_\nu(x)=-|\nabla w_r(x)|=-\frac{1}{N^{\frac{1}{p-1}}}r^{\frac{1}{p-1}}.$$
Hence, 
$$\frac{1}{N}(w_r)_\nu(x)+H(x)|(w_r)_\nu(x)|^{p}=0\quad\mbox{for every }x\in\partial B_r$$
and the inequality in (ii) is satisfied with the equality sign and we are done. 
We consider now the remaining case $P_\nu>0$ on $\partial\Omega$. In this case 
$$(p-1)|u_\nu|^{p-2}u_{\nu\nu}+\frac{1}{N}<0\quad\mbox{on }\partial\Omega$$
(cf. \eqref{Pnu} and remember that $u_\nu<0$ on $\partial\Omega$), or equivalently
$$u_{\nu\nu}<-\frac{|u_{\nu}|^{2-p}}{N(p-1)}\quad\mbox{on }\partial\Omega.$$
Hence, $u_{\nu\nu}<0$ on $\partial\Omega$, and so, when $1<p<2$, we get 
$$
\mathcal I_p(u)\le  -\frac{p(N-1)}{p-1}\int_{\partial \Omega}\left(\frac{1}{N}u_\nu+H|u_\nu|^p\right)d\sigma. 
$$

(iii) Since $u$ is a solution of \eqref{P}, by Divergence Theorem and H\"older's inequality we have
$$
\begin{aligned}
|\Omega|&=\int_\Omega -\Delta_p udx=-\int_\Omega\mathrm{div}(|\nabla u|^{p-2}|\nabla u)dx=-\int_{\partial\Omega}|\nabla u|^{p-2}\nabla u\cdot \nu d\sigma\\
&=\int_{\partial\Omega}|u_\nu|^{p-1}d\sigma\le \left(\int_{\partial\Omega}|u_\nu|^p d\sigma\right)^{\frac{p-1}{p}}|\partial\Omega|^{\frac{1}{p}}.
\end{aligned}
$$
By using the definition of $H_0$, the previous estimate reads as
$$
\left(\int_{\partial\Omega}|u_\nu|^p d\sigma\right)^{\frac{1}{p'}}\ge\frac{|\Omega|}{|\partial\Omega|^{\frac{1}{p}}}=\frac{|\partial\Omega|^{\frac{1}{p'}}}{N H_0}.
$$
Consequently, by H\"older's inequality,
$$-\int_{\partial\Omega}u_\nu d\sigma\le \|u_\nu\|_{L^p(\partial\Omega)}|\partial\Omega|^{\frac{1}{p'}}\le NH_0\left(\int_{\partial\Omega}|u_\nu|^pd\sigma\right)^{\frac{1}{p}+\frac{1}{p'}}=NH_0\int_{\partial\Omega}|u_\nu|^pd\sigma.
$$
By using this inequality, the right-hand side of \eqref{identity} can be estimated as
\begin{equation}\label{est_boundary}
-(N-1)\int_{\partial\Omega} \left(\frac{1}{N}u_\nu+H|u_\nu|^p\right) d\sigma\le (N-1)\int_{\partial\Omega}|u_\nu|^p(H_0-H)d\sigma.
\end{equation}

Therefore, in view of part (ii) of the present theorem, we have for $1<p<2$
$$
\mathcal I_p(u)\le \frac{p(N-1)}{p-1}\int_{\partial\Omega}|u_\nu|^p(H_0-H)d\sigma. 
$$
This concludes the proof.
\end{proof}

\begin{remark}
From parts (i) and (iii) of the previous theorem, since $|u_\nu|^p$ is bounded on $\partial\Omega$, we have the following upper bound for the $L^1$-norm of the mean curvature $H$ of $\partial\Omega$
$$\int_{\partial\Omega}Hd\sigma\le H_0|\partial\Omega|=\frac{|\partial\Omega|^2}{N|\Omega|}.$$
\end{remark}
\smallskip 

The previous theorem allows us to give an alternative proof of the Soap Bubble Theorem in the case in which the hypersurface is a level set of the solution of problem \eqref{P}. 

\begin{proof}[Proof of Theorem \ref{SBT}.]
The scheme of the proof is the following: a. $\Rightarrow$ c. $\Rightarrow$ b. $\Rightarrow$ c. $\Rightarrow$ a., this proves that a., b. and c. are all equivalent;  then we will prove that  a. $\Rightarrow$  d. $\Rightarrow$ c.,  and finally  b. $\Rightarrow$ e.
\smallskip 

\underline{a. $\Rightarrow$ c.}  If $\Omega=B_r$, the only solution of \eqref{P} is the radial function $w_r$ defined in \eqref{rad-sol}. 
\smallskip

\underline{c. $\Rightarrow$ b.}  As in the proof of Theorem \ref{th:main}-(ii),  if the solution of \eqref{P} is radial, $\Omega=B_r$ for some $r>0$, and so $u=w_r$. Hence, by strighforward calculations, b. holds true. \smallskip

\underline{b. $\Rightarrow$ c.}
By Theorem \ref{th:main}-(ii), we get $\mathcal I_p(u)=0$, which in turn implies that $u$ is radial, by Lemma \ref{onlyrad}. \smallskip

\underline{c. $\Rightarrow$ a.} If $u$ is radial, then $\Gamma=\partial\Omega$, being a level set of $u$, is a sphere, and so $\Omega$ is a ball.

\underline{a. $\Rightarrow$ d.} If $\Omega=B_r$ for some $r>0$, then $u=w_r$ and so, for every $x\in\partial\Omega$
\begin{equation}\label{H0R0}
H(x)=-\frac{1}{N-1}\mathrm{div}\frac{\nabla w_r}{|\nabla w_r|}=\frac{1}{N-1}\sum_{i=1}^N\left(\frac{1}{|x|}-\frac{x_i^2}{|x|^3}\right)=\frac{1}{r}=\frac{|\partial B_r|}{N|B_r|}=H_0.
\end{equation}
\smallskip

\underline{d. $\Rightarrow$ c.} By Theorem \ref{th:main}-(iii), we get $\mathcal I_p(u)=0$, which in turn implies that $u$ is radial, by Lemma \ref{onlyrad}. \smallskip 

\underline{b. $\Rightarrow$ e.} Up to now, we have proved that a., b., c. and d. are equivalent. Thus, if b. holds, we have by d. 
$$
|u_\nu|^{p-2}u_\nu=-\frac{1}{NH_0}\quad\mbox{on }\partial\Omega.
$$
We recall that, on $\partial\Omega$, $\nu=-\frac{\nabla u}{|\nabla u|}$ and consequently $u_\nu=\nabla u\cdot\nu=-|\nabla u|$.
Therefore, 
$$
|u_\nu|^{p-2}u_\nu=-|\nabla u|^{p-1}=-\frac{1}{NH_0}\quad\mbox{on }\partial\Omega,
$$ which gives e. 
\end{proof}

In the remaining part of this section, we give an upper bound of the integral $\mathcal I_p(u)$ in terms of the $L^1(\partial\Omega)$-norm of the difference between the mean curvature of $\partial\Omega$ and the reference constant $H_0$. We start with some preliminary results.

\begin{lemma}\label{annulus}
Let $\Omega=\mathcal A(R_1,R_2)$ be an annulus of radii $0<R_1<R_2$, then there exists a unique $\bar R\in (R_1,R_2)$ such that the positive radial function 
\begin{equation}\label{u_annulus}
u_\mathcal A(r):=\begin{cases}
\mathlarger{\int}_{R_1}^r\left(\displaystyle{\frac{\bar R^N}{N \tau^{N-1}}-\frac{\tau}{N}}\right)^{\frac1{p-1}}\,d\tau\quad&\mbox{for every }r\in [R_1,\bar R],\smallskip\\
 \mathlarger{\int}_r^{R_2}\left(\displaystyle{\frac{\tau}{N}-\frac{\bar R^N}{N\tau^{N-1}}}\right)^{\frac1{p-1}}\,d \tau&\mbox{for every }r\in (\bar R,R_2]
\end{cases}
\end{equation}
is of class $C^1([R_1,R_2])$ and solves \eqref{P}. Furthermore, $u_\mathcal A$ achieves its maximum at $\bar R$, where with abuse of notation we have written $u_\mathcal A(x)=u_\mathcal A(r)$ for $|x|=r$.
\end{lemma}
\begin{proof} Suppose first that such $\bar R$ exists and belongs to $(R_1,R_2)$. In this case, it is straightforward  to verify that the function $u_\mathcal A$ given in \eqref{u_annulus} solves problem \eqref{P}, which can be written in radial form as 
$$
\begin{cases}
|u_\mathcal A'|^{p-2}\left[(p-1)u_\mathcal A''+\frac{N-1}{r}u_\mathcal A'\right]=-1\quad&\mbox{in }(R_1,R_2),\\
u_\mathcal A(R_1)=u_\mathcal A(R_2)=0,&  
\end{cases}
$$
where the symbol $'$ denotes the derivative with respect to $r$. 

Now, if we consider the two functions 
$$
\begin{aligned}
&F_1:\,\rho\in[R_1,R_2]\mapsto \int_{R_1}^\rho \left(\frac{\rho^N}{N\tau^{N-1}}-\frac{\tau}{N}\right)^{\frac1{p-1}}\,d\tau\,\in\mathbb R,\\
&F_2:\,\rho\in[0,R_2]\mapsto \int_\rho^{R_2} \left(\frac{\tau}{N}-\frac{\rho^N}{N\tau^{N-1}}\right)^{\frac1{p-1}}\,d\tau\,\in\mathbb R,\\
\end{aligned}
$$
they have the following properties: 
$$
\begin{gathered}
F_1(R_1)=F_2(R_2)=0,\\
0<F_1(\rho)<+\infty\mbox{ for every }\rho\in(R_1,R_2], \quad 0<F_2(\rho)<+\infty\mbox{ for every }\rho\in[0,R_2), \\
F_1'(\rho)=\frac{1}{p-1}\int_{R_1}^\rho\left(\frac{\rho^N}{N\tau^{N-1}}-\frac{\tau}{N}\right)^{\frac{2-p}{p-1}}\left(\frac{\rho}{\tau}\right)^{N-1}\,d\tau>0\mbox{ for every }\rho\in (R_1,R_2],\\
F_2'(\rho)=-\frac{1}{p-1}\int_{\rho}^{R_2}\left(\frac{\tau}{N}-\frac{\rho^N}{N\tau^{N-1}}\right)^{\frac{2-p}{p-1}}\left(\frac{\rho}{\tau}\right)^{N-1}\,d\tau<0\mbox{ for every }\rho\in [0,R_2).
\end{gathered}
$$
Therefore, there exists a unique $\rho=\bar R\in (R_1,R_2)$ for which $F_1(\bar R)=F_2(\bar R)$. This concludes the proof.
\end{proof}

\begin{definition}\label{opt_rad}
A domain $\Omega\subset\mathbb R^N$ satisfies the {\it uniform interior and exterior touching sphere conditions}, and we denote with $\rho_i$ and $\rho_e$ the optimal interior and exterior radii respectively, if for any $x_0\in \partial\Omega$ there exist two balls $B_{\rho_i}(c^-)\subset\Omega$ and $B_{\rho_e}(c^+)\subset\mathbb R^N\setminus\bar\Omega$ such that $x_0\in\partial B_{\rho_i}(c^-)\cap \partial B_{\rho_e}(c^+)$. 
We call {\it optimal radius} the minimum between the interior and the exterior radius, $\rho:=\min\{\rho_i,\rho_e\}$.
\end{definition}

We observe that is $\Omega$ is of class $C^2$, then it satisfies the uniform interior and exterior touching sphere conditions.

\begin{proposition}\label{bdd_grad}
Let $\Omega\subset\mathbb R^N$ be a bounded domain of class $C^2$ and $u\in C^1(\bar\Omega)$ be a solution of \eqref{P} in $\Omega$. Then 
$$\left(\frac{\rho_i}{N}\right)^{\frac{1}{p-1}} \le |\nabla u|\le \left[\frac{(\mathrm{diam}(\Omega)+\rho_e)^N}{N\rho_e^{N-1}}-\frac{\rho_e}{N}\right]^{\frac{1}{p-1}}\quad\mbox{on }\partial\Omega.
$$
\end{proposition}
\begin{proof}
We follow the ideas in \cite[Theorem 3.10]{MP2016}.  Let $x_0$ be any point on the boundary $\partial\Omega$. Without loss of generality, we can place the origin at $c^-$. Thus, the function
$$u_{\rho_i}:=-\frac{p-1}{p N^{\frac1{p-1}}}\left(|x|^{\frac{p}{p-1}}-\rho_i^{\frac{p}{p-1}}\right)$$ is the solution of \eqref{P} in $B_{\rho_i}$.
Now,  being by definition $B_{\rho_i}\subset\Omega$,
$$
\begin{cases}
-\Delta_p u_{\rho_i}=-\Delta_p u\quad&\mbox{in } B_{\rho_i},\\
u_{\rho_i}\le u&\mbox{on }\partial B_{\rho_i},
\end{cases}
$$ 
and so, by comparison \cite[Lemma 3.7]{FGK2006}, $u_{\rho_i}\le u$ in $B_{\rho_i}$. Since $u_{\rho_i}(x_0)=u(x_0)$, we have $\partial_\nu(u_{\rho_i}-u)(x_0)>0$, where $\nu$ is the external unit normal to $B_{\rho_i}$. This gives the first inequality in the statement, namely
$$
|\nabla u(x_0)|\ge\left(\frac{\rho_i}{N}\right)^{\frac{1}{p-1}}.
$$
On the other hand, let $\mathcal A:=\mathcal A(\rho_e,\mathrm{diam}(\Omega)+\rho_e)$ be the annulus centered at $c^+$. By definition, $\Omega\subset\mathcal A$. Again, without loss of generality, we can place the origin at $c^+$ and consider the function $u_\mathcal A$ whose expression is given by \eqref{u_annulus} with $R_1:=\rho_e$ and $R_2:=\mathrm{diam}(\Omega)+\rho_e$. Reasoning as above we have
$$
\begin{cases}
-\Delta_p u_{\mathcal A}=-\Delta_p u\quad&\mbox{in } \Omega,\\
u_{\mathcal A}\ge u&\mbox{on }\partial \Omega,
\end{cases}
$$
and so $u_\mathcal A\ge u$ in $\Omega$. Therefore, $\partial_{\nu}(u_\mathcal A-u)(x_0)\le 0$, being $\nu$ the external unit normal to $\mathcal A$. This finally gives
$$
|\nabla u(x_0)|\le \left(\frac{\bar R^N}{N\rho_e^{N-1}}-\frac{\rho_e}{N}\right)^{\frac{1}{p-1}}\le \left(\frac{(\rho_e+\mathrm{diam}(\Omega))^N}{N\rho_e^{N-1}}-\frac{\rho_e}{N}\right)^{\frac{1}{p-1}}
$$
and concludes the proof.
\end{proof}

Combining together the results in Proposition \ref{bdd_grad} and Theorem \ref{th:main}, we get the following corollary.

\begin{corollary}\label{I-H}
Let $1<p<2$ and $\Omega\subset\mathbb R^N$ be a $C^{2,\alpha}$ bounded domain. If $u$ solves \eqref{P} and has $|\mathcal C|=0$, the following chain of inequalities holds
\begin{equation}\label{IpH}
0\le \mathcal I_p(u)\le \frac{p(N-1)}{p-1}\left[\frac{(\mathrm{diam}(\Omega)+\rho_e)^N}{N\rho_e^{N-1}}-\frac{\rho_e}{N}\right]^{\frac{p}{p-1}}\|H_0-H\|_{L^1(\partial\Omega)}. 
\end{equation}
\end{corollary}

\section{Some comments on the stability}\label{sec4}
With reference to the result given in Corollary~\ref{I-H}, we observe that, while $\mathcal I_p(u)$ is related to the solution of problem \eqref{P}, the constant that bounds from above $\mathcal I_p(u)$ in \eqref{IpH} depends only on the geometry of the problem. In particular, the non-negative integral $\mathcal I_p(u)$ that vanishes only on radial functions, goes to zero as $H\to H_0$ in $L^1(\partial\Omega)$. In view of Corollary~\ref{I-H}, this suggests, at least qualitatively, a sort of stability of the Serrin-type result for the overdetermined problem with the $p$-Laplacian.  

In \cite{CV2016}, Ciraolo and Vezzoni obtained the following stability result for the Soap Bubble Theorem by Alexandrov. 

\begin{theorem}[Theorem 1.1 of \cite{CV2016}]\label{th:CV}
Let $\partial\Omega$ be a $C^2$-regular, connected, and closed hypersurface embedded in $\mathbb R^N$. If 
$$
\|H-H_0\|_{L^\infty(\partial\Omega)}<\varepsilon
$$ 
for some $\varepsilon>0$ depending only on $N$, $|\partial\Omega|$, and upper bounds on the inverse of the optimal radius (cf. Definition \ref{opt_rad}) $\rho^{-1}$ of $\partial\Omega$, then 
$\partial\Omega\subset \bar B_{r_e}\setminus B_{r_i}$, with 
$$0<r_e-r_i\le C \varepsilon,$$
where $C>0$ depends on $N$, $|\partial\Omega|$, and upper bounds on the inverse of the optimal radius $\rho^{-1}$ of $\partial\Omega$. 
\end{theorem}

This result gives an estimate of $r_e-r_i$ in terms of the $L^\infty(\partial\Omega)$-norm of $H-H_0$.

 Furthermore, as a consequence, for every $1<p<\infty$, it is possible to compare the solution $u$ of \eqref{P} with the radial solutions
$$
\begin{gathered}
u_e(x):=-\frac{p-1}{pN^{\frac1{p-1}}}\left(|x|^{\frac{p}{p-1}}-(r_e)^{\frac{p}{p-1}}\right)\quad\mbox{for every }x\in B_{r_e}\\
\mbox{ and }\\
u_i(x):=-\frac{p-1}{pN^{\frac1{p-1}}}\left(|x|^{\frac{p}{p-1}}-(r_i)^{\frac{p}{p-1}}\right)\quad\mbox{for every }x\in B_{r_i}
\end{gathered}
$$ 
of 
$$
\begin{cases}
-\Delta_p u_e=1\quad&\mbox{in }B_{r_e},\\
u_e=0&\mbox{on }\partial B_{r_e},
\end{cases}
\quad\mbox{and}\quad
\begin{cases}
-\Delta_p u_i=1\quad&\mbox{in }B_{r_i},\\
u_i=0&\mbox{on }\partial B_{r_i},
\end{cases}
$$
respectively. 
Indeed, by the weak comparison principle \cite[Lemma 3.7]{FGK2006}, we easily get 
$$
u\ge u_i\;\;\mbox{in }B_{r_i}\quad\mbox{and}\quad u\le u_e\;\;\mbox{in }\Omega,
$$
giving in particular the following estimate of $u$ in terms of the radial solutions $u_i$ and $u_e$ on the interior ball $B_{r_i}$ 
$$
-\frac{p-1}{pN^{\frac1{p-1}}}\left(|x|^{\frac{p}{p-1}}-(r_i)^{\frac{p}{p-1}}\right)\le u(x)\le -\frac{p-1}{pN^{\frac1{p-1}}}\left(|x|^{\frac{p}{p-1}}-(r_e)^{\frac{p}{p-1}}\right)\quad\mbox{in } B_{r_i}.
$$

It is quite challenging to obtain an estimate from below of $\mathcal I_p(u)$ in terms of some increasing function of $r_e-r_i$. This would allow to improve --at least in some relevant cases-- the stability result in Theorem \ref{th:CV}, getting a stability result in terms of the $L^1(\partial\Omega)$-norm, instead of the $L^1(\partial\Omega)$-norm, of $H-H_0$. This approach was proposed by Magnanini and Poggesi for the case $p=2$ in \cite{MP2016}, where the authors used in a very clever way the mean value property for harmonic functions. Nevertheless, their method works well only in the linear case and seems very difficult to generalize it to the case $p\neq 2$. Some other issues related to the stability of the symmetry result for the overdetermined $p$-Laplacian problem are treated in \cite{Colesanti}.

\section*{Acknowledgments}
\noindent The authors were supported by the INdAM - GNAMPA Project 2017 ``Regolarit\`a delle soluzioni viscose per equazioni a derivate parziali non lineari degeneri" and by University of Bologna. The first author acknowledges also the support of the project ``Ricerca Locale 2018 Linea B - Problemi non lineari'' from University of Turin.

\bibliographystyle{abbrv}
\bibliography{biblio}

\end{document}